\documentclass[11pt]{amsart}
\usepackage[utf8]{inputenc}

\usepackage[dvipsnames]{xcolor}
\usepackage{geometry}
\usepackage{chemarrow}
\usepackage{tikz}

\usepackage{amssymb,amsmath,amsthm}
\usepackage[utf8]{inputenc}
\usepackage{amsthm}
\usepackage{lineno}
\usepackage{float}
\usepackage{enumitem}
\usepackage{tabu}
\usepackage{nicematrix}
\usepackage{resizegather}
\usepackage{bbding}
\usepackage{tikz}
\usepackage{dsfont}
\usepackage{csquotes}
\usepackage{graphicx}
\usepackage{subcaption}
\usepackage{mathtools}
\usepackage{comment}
\usepackage{graphicx}
\usepackage{calc}
\usepackage{mathrsfs}
\usepackage{bm}
\mathtoolsset{showonlyrefs=true}
\definecolor{cornellred}{rgb}{0.7, 0.11, 0.11}
\usepackage[hidelinks]{hyperref}

\DeclareMathOperator{\Fr}{Fr}
\DeclareMathOperator{\Col}{Col}
\DeclareMathOperator{\Aut}{Aut}
\DeclareMathOperator{\Char}{char}
\DeclareMathOperator{\PGL}{PGL}
\DeclareMathOperator{\GL}{GL}
\DeclareMathOperator{\GammaL}{\Gamma L}
\DeclareMathOperator{\PGammaL}{P\Gamma L}

\newcommand{\op}{
  \mathop{
    \vphantom{\bigoplus} 
    \mathchoice
      {\vcenter{\hbox{\resizebox{\widthof{$\displaystyle\bigoplus$}}{!}{$\boxplus$}}}}
      {\vcenter{\hbox{\resizebox{\widthof{$\bigoplus$}}{!}{$\boxplus$}}}}
      {\vcenter{\hbox{\resizebox{\widthof{$\scriptstyle\oplus$}}{!}{$\boxplus$}}}}
      {\vcenter{\hbox{\resizebox{\widthof{$\scriptscriptstyle\oplus$}}{!}{$\boxplus$}}}}
  }\displaylimits 
}

\hypersetup{
    pdfpagemode={UseOutlines},
    bookmarksopen,
    pdfstartview={FitH},
    colorlinks,
    linkcolor={RoyalBlue},
    citecolor={cornellred},
    urlcolor={ForestGreen}
            }

\def\N {\mathbb N}

\definecolor{lime}{HTML}{A6CE39}
\DeclareRobustCommand{\orcidicon}{%
	\begin{tikzpicture}
	\draw[lime, fill=lime] (0,0) 
	circle [radius=0.16] 
	node[white] {{\fontfamily{qag}\selectfont \tiny ID}};
	\draw[white, fill=white] (-0.0625,0.095) 
	circle [radius=0.007];
	\end{tikzpicture}
	\hspace{-2mm}
}

\foreach \x in {A, ..., Z}{%
	\expandafter\xdef\csname orcid\x\endcsname{\noexpand\href{https://orcid.org/\csname orcidauthor\x\endcsname}{\noexpand\orcidicon}}
}

\newtheorem{theorem}{Theorem}[section]

\newtheorem{proposition}[theorem]{Proposition}
\newtheorem{corollary}[theorem]{Corollary}
\newtheorem{lemma}[theorem]{Lemma}
\newtheorem{fact}[theorem]{Fact}
\theoremstyle{definition}
\newtheorem{definition}[theorem]{Definition}
\newtheorem{example}[theorem]{Example}

\theoremstyle{remark}
\newtheorem{remark}[theorem]{Remark}
\newtheorem{claim}{Claim}

\title{On projective spaces over local fields}
\author[N.\ Cangiotti and A.\ Linzi]{Nicolò Cangiotti\orcidN{} \and Alessandro Linzi\orcidA{}}
\date{}

\address[N.\,Cangiotti]{Department of Mathematics\newline\indent Politecnico di Milano \newline\indent
via Bonardi 9, Campus Leonardo, 20133 Milan, Italy}
\email{nicolo.cangiotti@polimi.it}

\address[A.\,Linzi]{Center for Information Technologies and Applied Mathematics\newline\indent  University of Nova Gorica
\newline\indent Vipavska 13, Rožna Dolina, SI-5000 Nova Gorica, Slovenia}
\email{alessandro.linzi@ung.si}



\thanks{The authors are very grateful to A.\ Nobile, M.\ Stecconi, and N.\ Turtulici for fruitful discussions about cardinality and topology.}
\thanks{NC is member and acknowledge the support of
Gruppo Nazionale per l’Analisi Matematica, la Probabilit\`a e le loro Applicazioni
(GNAMPA) of Istituto Nazionale di Alta Matematica (INdAM). Moreover, NC is supported by the MIUR - PRIN 2017 project \lq\lq From
Models to Decisions\rq\rq\ (Prot. N. 201743F9YE)}

\begin{document}

\maketitle
\begin{abstract}
Let $\mathcal{P}$ be the set of points of a finite-dimensional projective space over a local field $F$, endowed with the topology $\tau$ naturally induced from the canonical topology of $F$. Intuitively, continuous incidence abelian group structures on $\mathcal{P}$ are abelian group structures on $\mathcal{P}$ preserving both the topology $\tau$ and the incidence of lines with points. We show that the real projective line is the only finite-dimensional projective space over an Archimedean local field which admits a continuous incidence abelian group structure. The latter is unique up to isomorphism of topological groups. In contrast, in the non-Archimedean case we construct continuous incidence abelian group structures in any dimension $n\in\N$. We show that if $n>1$ and $\Char(F)$ does not divide $n+1$, then there are finitely many possibilities up to topological isomorphism and, in any case, countably many. 
\end{abstract}
\section{Introduction}
It has been almost 90 years since the 8th Congress of Scandinavian Mathematicians, where
Frédéric Marty first introduced the notion of hypergroup \cite{Mar34}, providing also some applications on group theory as well as in the study of rational functions. In the following years many authors started to investigate this new mathematical object as, for instance,   Marc Krasner \cite{Kra37,Kra83} and Walter Prenowitz \cite{Pre43}. Nowadays, hypercompositional algebra is widely studied and its applications affect different fields of mathematics \cite{CL03}. In this manuscript, we focus on an algebraic description of geometric objects (such as lines) using multivalued operations and study the resulting algebraic structures. Recently, Alain Connes and Caterina Consani in \cite{CC11} have found the formalism of Krasner hyperrings and hyperfields \cite{Kra57,Kra83} useful to describe the algebra of the adèle class space of global fields. In their paper, some interesting aspects about the connection between some hyperstructures and incidence groups on projective geometries are also outlined. Motivated by these studies we propose a detailed analysis on incidence group structures on finite-dimensional projective spaces over local fields, which also preserve the natural induced topology. 

\medskip

The paper is organized as follows. We introduce in Section \ref{Sec2} some important concepts about local fields, projective spaces, and hyperstructures, recalling useful classical results. In Section \ref{Sec3} we lay out an in-depth description of $\mathbf{K}$-vector spaces associated to the projective geometries, focusing on their algebraic structure theory. Section \ref{Sec4} is devoted to the study of the multiplicative structures, naturally obtained from the descriptions of the previous section, that
preserve the standard topological and geometrical structure of the corresponding finite-dimensional
projective spaces. In Section \ref{Sec5} and Section \ref{Sec6} we complete the picture by investigating separately two distinct scenarios, the one dimensional case and the higher finite-dimensional case. This case distinction is needed since the Fundamental Theorem of Projective Geometry (Fact \ref{FTPr}) is not applicable in dimension one, where the only limitations are provided by the topological structure. In Appendix \ref{app1} and Appendix \ref{appB} we provide some more details on the well-known theory of local fields necessary for our investigation and we provide a thorough review on the topology induced by a local field on finite-dimensional projective spaces constructed over it.

\section{Preliminaries}
\label{Sec2}

\subsection{Local fields}
We recall the following well-known notions in order to fix the terminology that we shall use. 
\begin{itemize}[leftmargin=2cm]
\item[(HA)] A topological space $(X,\tau)$ is \emph{Hausdorff} if for all distinct $x,y\in X$ there exist $O_x,O_y\in\tau$ such that $x\in O_x$, $y\in O_y$ and $O_x\cap O_y=\emptyset$.
\item[(LC)] We call a topological space $(X,\tau)$ is \emph{locally compact} if for every $x\in X$ there is $O_x\in\tau$ and $C_x\subseteq X$ compact in $(X,\tau)$ such that $x\in O_x\subseteq C_x$.
\item[(SG)] A \emph{semitopological group} is a group $(G,o,i,e)$ endowed with a topology with respect to which the maps $o_a:G\to G$ defined as $o_a(x):=o(a,x)$ are continuous, for all $a\in G$.
\item[(TG)] A \emph{topological group} is a group $(G,o,i,e)$ endowed with a topology with respect to which the operation $o:G\times G\to G$ and the inverse function $i:G\to G$ are continuous, where $G\times G$ is endowed with the product topology.
\item[(TV)] A \emph{topological vector space} over a field $k$ is an $k$-vector space $V$ endowed with a topology with respect to which $(V,+)$ is a topological group and scalar multiplication by any $a\in k$ is a continuous function $V\to V$. 
\item[(TF)] A \emph{topological field} $F$ is a field endowed with a topology with respect to which both $(F,+,-,0)$ and $(F^\times,\cdot,^{-1},1)$ are topological (abelian) groups. 
\end{itemize}

\begin{definition}
A \emph{local field} is a topological field $(F,\tau)$ that is Hausdorff, locally compact and its topology is not discrete, i.e., $\tau$ is not the full power set of $F$.
\end{definition}

Let $F$ be a local field. By Theorem 5 and Theorem 8 in \cite{Wei74}, there are three mutually exclusive possibilities:

\begin{itemize}[leftmargin=2cm]
    \item[(AR)] $F=\mathbb{R}$ or  $F=\mathbb{C}$, with the Euclidean topology.
    \item[(NAR$_0$)] $F$ is a finite extension of $\mathbb{Q}_p$, for some prime $p$, with the topology induced by the $p$-adic norm.
    \item[(NAR$_p$)] $F=\mathbb{F}_q((t))$, where $q=p^n$ for some $n\in\N$, with the topology induced by the $t$-adic norm ($p=\Char F>0$).
\end{itemize}

These possibilities are usually referred to as the \emph{Archimedean} case, the \emph{non-Archimedean, characteristic zero} case and the \emph{non-Archimedean, positive characteristic} case, respectively.

We have collected in Appendix \ref{app1} more details on the structure and the topology of local fields, in the various cases listed above.
\begin{comment}
Local fields are studied in number theory and algebraic geometry as they arise as completions of global fields, i.e., algebraic number fields and function fields of algebraic curves over finite fields.
\end{comment}

\subsection{Projective spaces}

For a field $K$, we will consider the \emph{projective space of dimension }$n\in\mathbb{N}$ \emph{over}~$K$: 
\[
\mathbb{P}_K^n:=K^{n+1}\setminus\{\bar 0\}/K^\times.
\]
By this we mean that $\mathbb{P}_K^n$ is the quotient set of $K^{n+1}\setminus\{\bar 0\}$ under the orbit equivalence relation $\sim$ induced by the orbits of the action of $K^\times$ by scalar multiplication:
\[
\bar x\sim\bar y\iff \exists a\in K^\times:\bar y=a\bar x \quad\quad (\bar x,\bar y\in K^{n+1}\setminus\{\bar 0\}).
\]
We shall denote by $\hat x$ or $[\bar x]$ the $\sim$-equivalence class of $\bar x\in K^{n+1}\setminus\{\bar 0\}$ in $\mathbb{P}_K^n$.

If $\mathcal{P}:=\mathbb{P}_K^n$, then for all distinct $\hat x,\hat y\in\mathcal{P}$ we will consider the following subsets of $\mathcal{P}$:
\begin{equation}\label{lin}
\ell(\hat x,\hat y):=\{\hat z\in\mathcal{P}\mid \exists a,b\in K^\times: \bar z=a\bar x+b\bar y \}\cup\{\hat x,\hat y\}.
\end{equation}

Clearly, the above sets do not depend on the choice of the representatives $\bar x,\bar y,\bar z$ of $\hat x$, $\hat y$ and $\hat z$ since if $\exists a,b\in K^\times: \bar z=a\bar x+b\bar y$ holds for some choice of representatives, then it holds for any such choice.\par 

We consider two projective spaces $\mathbb{P}(V)$ and $\mathbb{P}(W)$ over a field $K$ \emph{isomorphic} if $V$ and $W$ are isomorphic as $K$-vector spaces. Thus, $K^{n+1}$ can be replaced in the above discussion with any $K$-vector space $V$ of the same dimension. The projective space associated to a $K$-vector space $V$ will be denoted by $\mathbb{P}(V)$.

It can be verified that $(\mathcal{P},\mathcal{L})$ is a \emph{projective geometry}, i.e., satisfies the following axioms:
\begin{itemize}[leftmargin=2cm]
\item[(P0)] $\mathcal{L}$ is a collection of subsets of $\mathcal{P}$.
    \item[(P1)] For all distinct $x,y\in\mathcal{P}$ there exists a unique \emph{line} $\ell(x,y)\in\mathcal{L}$ such that $x,y\in\ell(x,y)$.
    \item[(P2)] For all distinct $x,y\in\mathcal{P}$, all $z\in\mathcal{P}\setminus\ell(x,y)$, all $t\in\ell(x,y)\setminus\{x\}$ and all $u\in\ell(x,z)\setminus\{x\}$ we have that $\ell(y,z)\cap\ell(t,u)\neq\emptyset$.
    \item[(P3)] Every line contains at least $4$ points.\footnote{The original axiom postulates that any line has at least $3$ points. However, our choice makes no difference makes no difference as soon as we exclude the field $\mathbb{F}_2$ as base field in our investigation.}
\end{itemize}
For a projective geometry $(\mathcal{P},\mathcal{L})$, we say that a line $\ell\in\mathcal{L}$ is \emph{incident} with a point $x\in\mathcal{P}$ whenever $x\in\ell$. 

\begin{definition}
A projective geometry $(\mathcal{P},\mathcal{L})$ is called \emph{Desarguesian} if the following holds:
\begin{itemize}[leftmargin=2cm]
    \item[(DS)] For all pairwise distinct $x_1,x_2,x_3,y_1,y_2,y_3\in\mathcal{P}$ such that
    \begin{itemize}
        \item[-] $\bigcap_{i=1,2,3}\ell(x_i,y_i)=\{z\}$,
        \item[-] $c\notin\ell(a,b)$, for any subset $\{a,b,c\}$ of $\{x_1,x_2,x_3,z\}$ or of $\{y_1,y_2,y_3,z\}$ with cardinality $3$,
    \end{itemize}
we have that the $3$ points uniquely determined by 
\[
z_{ij}\in\ell(x_i,x_j)\cap\ell(y_i,y_j)\quad (1\leq i\neq j\leq 3)
\]
satisfy $z_{31}\in\ell(z_{12},z_{23})$.
\end{itemize}
\end{definition}

In a (finitely generated) projective geometry $(\mathcal{P},\mathcal{L})$ one can define the notion of dimension as explained in e.g.\ \cite[p.\ 18]{BR98}. The following is a well-known theorem:

\begin{fact}[Theorem 3.4.2 and Corollary 3.4.3 in \cite{BR98}]
For any finite-dimensional Desarguesian projective geometry of dimension at least $2$ there exists a field $k$ and a $k$-vector space $V$ such that $\mathcal{P}=\mathbb{P}(V)$ and the lines in $\mathcal{L}$ are obtained exactly as in \eqref{lin}. The Desarguesianity hypothesis can be dropped if one assumes the dimension to be at least $3$.
\end{fact}

In dimension $2$ there are projective geometries which are not \emph{realised} as projective spaces over a field, these are known as non-Desarguesian planes and have been extensively studied (see e.g.\ \cite[Chapter XX]{Hal59}).

\begin{remark}\label{projrel}
The above description is of interest since a projective geometry, as defined above, is a first-order structure over the language with one ternary relation symbol $\mathfrak{r}$ interpreted as:
\[
\mathfrak{r}(x,y,z)\iff x\neq y\text{ and }z\in\ell(x,y).
\]
\end{remark}

\begin{remark}\label{Conrad}
The following observation is essentially a consequence of Tychonoff Theorem (for finite products\footnote{In particular, we are not making use of the axiom of choice here.}). If $F$ is a local field and $\tau$ the topology induced by its canonical norm, then any $F$-vector space of dimension $n\in\N$ is homeomorphic to $F^n$ (endowed with the product topology induced by $\tau$). More details on this fact are discussed e.g.\ in \cite[Section 4]{Conr}.

Thus, any projective space $\mathbb{P}(V)$, where $V$ a vector space of dimension $n+1$ over a local field $F$, is isomorphic and homeomorphic to $\mathbb{P}^n_F$.
\end{remark}

For $(F,\tau)$ as in the above remark, we shall call the quotient topology on $\mathbb{P}_F^n$ the \emph{standard topology} on the $n$-dimensional projective space over $F$. In Appendix \ref{appB} below, we provide a proof of the following well-known fact.

\begin{fact}\label{Pcomp}
All finite-dimensional projective spaces $\mathbb{P}_F^n$ over a local field $F$ are Hausdorff and compact with respect to their standard topology.
\end{fact}

For $V=F^\omega$, i.e., the space of infinite sequences of elements of $F$, we will write $\mathbb{P}_F^\omega$ in place of $\mathbb{P}(F^\omega)$ by analogy with the finite-dimensional case, but we postpone the discussion on the topology of this infinite-dimensional projective space to future investigations.

\subsection{Hypercompositional algebra}

Let $H$ be a non-empty set and $\mathit{2}^H$ its power set. A \emph{multivalued operation} $\boxplus$ on $H$ is a function which associates to every pair $(x,y) \in H \times H$ an element of $\mathit{2}^H$, denoted by $x\boxplus y$. If $\boxplus$ is a multivalued operation on $H\neq\emptyset$, then for $x \in H$ and $A,B\subseteq H$ we set 
\[
A\boxplus B:=\bigcup_{(a,b)\in A\times B} a\boxplus b,
\]
$A \boxplus  x := A \boxplus  \lbrace x \rbrace$ and $x \boxplus A := \lbrace x \rbrace \boxplus  A$. If $A$ or $B$ is empty, then so is $A\boxplus B$.\par
A \emph{hypergroup} can be defined as a non-empty set $H$ with a multivalued operation $\boxplus $ which is associative (see Definition \ref{hypergp} (CH1) below) and \emph{reproductive on $H$} (i.e., $x \boxplus  H = H \boxplus  x = H$ for all $x \in H$).
This notion was introduced by F.\ Marty in \cite{Mar34,Mar35,Mar36}. The theory of hypergroups, with a detailed historical overview, is presented in \cite{Mas21}, where an extensive bibliography is also provided.
\begin{definition}
A \emph{hyperoperation} $\boxplus $ on a non-empty set $H$  is a multivalued operation such that $x\boxplus y\neq\emptyset$ for all $x,y\in H$.
\end{definition} 
\begin{lemma}[Theorem 12 in \cite{Mas21}]\label{x+ynempty}
If $(H,\boxplus)$ is a hypergroup, then $\boxplus$ is a hyperoperation~on~$H$. 
\end{lemma}

The following special class of hypergroups will be of interest for us.

\begin{definition}[\cite{Kra57}, Definition 2 in \cite{Lin23}] \label{hypergp}
A \emph{canonical hypergroup} is a triple $(H,\boxplus ,0)$, where $H\neq\emptyset$, $\boxplus$ is a multivalued operation on $H$ and $0$ is an element of $H$ such that the following axioms hold:
\begin{itemize}[leftmargin=2cm]
\item[(CH1)] $\boxplus$ is associative, i.e., $(x\boxplus y)\boxplus z=x\boxplus (y\boxplus z)$ for all $x,y,z \in H$,
\item[(CH2)] $x\boxplus y=y\boxplus x$ for all $x,y\in H$,
\item[(CH3)] for every $x\in H$ there exists a unique $y\in H$ such that $0\in x\boxplus y$. The element $y$ will be denoted by $x^-$.
\item[(CH4)] $z\in x\boxplus y$ implies $y\in z\boxplus x^-$ for all $x,y,z\in H$.
\end{itemize}
\end{definition}

\begin{remark}
Let $(H,\boxplus,0)$ be a canonical hypergroup. The hyperoperation $\boxplus$ can be encoded in a first-order language using a ternary relation
\[
\mathfrak{h}(x,y,z)\iff z\in x\boxplus y\quad\quad(x,y,z\in H).
\]
This observation will be of interest in view of Remark \ref{projrel} above.
\end{remark}

Let us now recall the definition of (Krasner) hyperrings and hyperfields.

\begin{definition}[\cite{Kra57,Kra83}, Definition 3 in \cite{Lin23}]\label{Krasnerhyp}
A \emph{(commutative) hyperring} is a tuple $(R,\boxplus ,\cdot,0)$ which satisfies the following axioms:
\begin{itemize}[leftmargin=2cm]
\item[(HR1)] $(R,\boxplus ,0)$ is a canonical hypergroup,
\item[(HR2)] $(R,\cdot)$ is a (commutative) semigroup and $0\cdot x=x\cdot0= 0$, for all $x\in R$,
\item[(HR3)] the operation $\cdot$ is distributive with respect to $\boxplus$. That is, for all $x,y,z\in R$,
\[
x(y\boxplus z)=xy\boxplus xz,
\]
\end{itemize}
where, for $x\in R$ and $A\subseteq R$, we have set
\[ 
xA:=\{xa\mid a\in A\}.
\]
If $xy=0$ implies $x=0$ or $y=0$, then $(R,\boxplus ,\cdot,0)$ is called an integral hyperdomain. If $(R,\cdot)$ is a monoid with neutral element $1 \neq 0$, then we say that $(R,\boxplus ,\cdot,0,1)$ is a \emph{hyperring with unity}.
If  $(R,+,\cdot,0,1)$ is a hyperring with unity and 
$(R\setminus\{0\},\cdot,1)$ is an abelian group, then $(R,+,\cdot,0,1)$ is called a \emph{hyperfield}.
If $R$ is a hyperring with unity, then we denote by $R^\times$ the multiplicative group of the \emph{units} of $R$, i.e.,
\[
R^\times:=\{x\in R\mid\exists y\in R~:~xy=1\}.
\]
In particular, if $R$ is a hyperfield, then $R^\times=R\setminus\{0\}$.
\end{definition}

Since we will only consider the commutative case, in the following sections we will call commutative hyperrings simply hyperrings.\par

\begin{example}[\cite{Kra83}]\label{Kraquot}
Let $A$ be a commutative ring and $T$ a subgroup of the commutative semigroup $(A,\cdot)$. Let $A_T$ denote the set of all multiplicative cosets $xT$ for $x\in A$, in particular $0T=\{0\}$. Krasner showed that by setting
\[
xT\boxplus yT:=\{zT\in A_T\mid \exists t,s\in T:z=xt+ys\}
\]
and
\[
xT\cdot yT:=xyT
\]
we obtain that the structure $(A_T,\boxplus,\cdot,0T)$ is a hyperring where $(xT)^-=(-x)T$ for all $x\in A$. In addition, if $A$ is an integral domain (resp.\ field), then this construction always yields an integral hyperdomain (resp.\ hyperfield). This is called the \emph{factor hyperring (or hyperfield)} of $A$ modulo $T$.
\end{example}

Let us make a few basic observations about factor hyperrings and hyperfields.

\begin{lemma}
Let $A$ be an integral domain with $1$ and $T$ a subgroup of $(A,\cdot,1)$. Then $A_T$ is a hyperfield (with respect to the operations defined in Example \ref{Kraquot} above) if and only if $A$ is a field.
\end{lemma}

\begin{proof}
One implication was already shown to hold by Krasner in \cite{Kra83}. For the other implication, assume that $A_T$ is a hyperfield and for $x\in A\setminus\{0\}$ find the inverse $yT$ of $xT$ in $A_T$. Then $(xy)T=xT\cdot yT=1T$ implies that there exists $t\in T$ such that $xyt=1$, so $yt=x^{-1}$ and $A$ is a field.    
\end{proof}

\begin{lemma}\label{x+x=0x}
Let $A$ be a commutative ring with $1$ and $T$ a subgroup of $(A\setminus\{0\},\cdot,1)$. Then the following are equivalent:
\begin{enumerate}
    \item[$(i)$] $T$ is the multiplicative subgroup of a subfield of $A$.
    \item[$(ii)$] $1T\boxplus 1T=\{0T,1T\}$.
\end{enumerate}
\end{lemma}
\begin{proof}
Assume $(i)$. Then $t+s\in T\cup\{0\}$ for all $t,s\in T$ and by definition, for all $a\in A$ we have that
\begin{equation}\label{subf}
1T\boxplus 1T=\{bT\in A_T\mid \exists t,s\in T : b=t+s\}=\{0T,1T\}.
\end{equation}
This shows $(ii)$.\par
For the converse, note that form $(ii)$ it follows that
\[
\{0T,1T\}=1T\boxplus 1T=\{(t-s)T\mid t,s\in T\}.
\]
Therefore, $T-T\subseteq T$, i.e., $(T\cup\{0\},+,0)$ is a subgroup of $(A,+,0)$, which shows $(i)$.
\end{proof}

Our interest for factor hyperrings satisfying property $(ii)$ of the above lemma is motivated by the following result.

\begin{proposition}[Proposition 3.1 in \cite{CC11}, see also \cite{Lyn61,Pre43}]\label{CC}
 Let $(\mathcal{P},\mathcal{L})$ be a projective geometry and $0$ be a point that is not already in $\mathcal{P}$. Then $H=H(\mathcal{P},\mathcal{L}):=\mathcal{P}\cup\{0\}$ endowed with the hyperoperation $\boxplus$ defined, for $x,y\in\mathcal{P}$, as
\[
x\boxplus y:=\begin{cases}\ell(x,y)\setminus\{x,y\}&\text{if }x\neq y,\\ \{0,x\}&\text{otherwise}\end{cases}
\]
and $x\boxplus 0=0\boxplus x:=\{x\}$, for all $x\in H$, is a canonical hypergroup with neutral element $0$.

Conversely, if $(H,\boxplus,0)$ is a canonical hypergroup satisfying $x\boxplus x=\{0,x\}$ for all $x\in H$, then, by setting $\mathcal{P}:=H^\times$,
\[
\ell(x,y):=x\boxplus y\cup\{x,y\}
\]
and $\mathcal{L}(H):=\{\ell(x,y)\mid x,y\in\mathcal{P}(H), x\neq y\}$ one obtains that $(\mathcal{P}(H),\mathcal{L}(H))$ is a projective geometry.
\end{proposition}

\begin{example}\label{realprlin}
The canonical hypergroup associated to the projective geometry realised as the real projective line $\mathbb{P}_\mathbb{R}^1$ is isomorphic to the additive hypergroup of the factor hyperfield $H:=\mathbb{C}_{\mathbb{R}^\times}$. Indeed, in the factor hyperfield $\mathbb{C}_{\mathbb{R}^\times}$ we have
\[
x\mathbb{R}^\times\boxplus y\mathbb{R}^\times=\begin{cases}\mathbb{C}^\times/\mathbb{R}^\times\setminus\{x\mathbb{R}^\times,y\mathbb{R}^\times\}&\text{if }x\mathbb{R}^\times\neq y\mathbb{R}^\times\\
\{0,x\mathbb{R}^\times\}&\text{if }x\mathbb{R}^\times=y\mathbb{R}^\times\end{cases}
\]
for all $x,y\in\mathbb{C}^\times$. Hence, the associated projective geometry has a unique line. Clearly, any element $(a+ib)\mathbb{R}^\times\in H^\times$ with $b\neq 0$ corresponds to a unique point $(ab^{-1}:1)$ of the real projective line. While real numbers correspond to the point $(1:0)$ of the real projective line.
\end{example}

For a local field $F$ and $n\in\N$ we will denote by $\mathbb{H}_F^n$ the canonical hypergroup obtained from the above proposition from the projective geometry of $\mathbb{P}_F^n$. After endowing $\mathbb{P}_F^n$ with the standard topology, the extension of the set of points with the new point $0$, described in the above proposition, suggests to look at the topology $\tau_\mathcal{A}$ on $\mathbb{H}_F^n$ obtained as the \emph{Alexandroff extension} \cite{Ale24} of $(\mathbb{P}_F^n,\tau)$, where $\tau$ denotes the standard topology. By definition a subset $O\subseteq\mathbb{H}_F^n$ is open in this topology if and only if $O\subseteq \mathbb{P}_F^n$ is open in $\mathbb{P}_F^n$ or $O\setminus\{0\}$ is the complement of a closed and compact set in $\mathbb{P}_F^n$. Note that, since by Fact \ref{Pcomp} we have that $\mathbb{P}_F^n$ is compact (and closed) we obtain that $\{0\}$ is an open and closed set in $\mathbb{H}_F^n$. 

Let us capture the structure that we have just described in the following general definition.

\begin{definition}
Let $(H,\boxplus,\cdot,0,1)$ be a hyperfield and assume that the multiplicative group $H^\times$ is a topological group with respect to some topology $\tau$. Then $H=H^\times\cup\{0\}$ endowed with the Alexandroff extension $\tau_\mathcal{A}$ of $\tau$ is called a \emph{topological hyperfield}. \par
We say that a topological hyperfield $H$ is \emph{Hausdorff, compact\ldots} whenever the same property holds for the topological group $H^\times$.
\end{definition}

\subsection{\textbf{K}-vector spaces}

Let $\mathbf{K}:=\{\mathbf{0},\mathbf{1}\}$ be endowed with the hyperoperation $\mathbf{1}\bm{+} \mathbf{1}=\{\mathbf{0},\mathbf{1}\}$, $\mathbf{0}\bm{+}\mathbf{0}=\{\mathbf{0}\}$ and $\mathbf{1}\bm{+}\mathbf{0}=\mathbf{0}\bm{+}\mathbf{1}=\{\mathbf{1}\}$. Then, with the natural multiplication, $\mathbf{K}$ is a hyperfield.\par
If $(H,\boxplus,0)$ is a canonical hypergroup, then we can interpret scalar multiplication by $\mathbf{1}$ as the identity $H\to H$ and scalar multiplication by $\mathbf{0}$ as the constant $0$ map $H\to H$. The action $\mathbf{K}\times H\to H$ thus obtained satisfies all the properties that a scalar multiplication has in the case in which $H$ is a group, except for $(a\bm{+}b)x=ax\boxplus bx$ ($a,b\in\mathbb{K}$, $x\in H$) which is equivalent to $x\boxplus x=\{0,x\}$ for all $x\in H$. This observation motivates the following definition.

\begin{definition} 
Let $(H,\boxplus,0)$ be a canonical hypergroup.
\begin{itemize}[leftmargin=2cm]
\item[$(i)$] If $x\boxplus x=\{0,x\}$ for all $x\in H$, then we say that $H$ is a $\mathbf{K}$-vector space.\par
\item[$(ii)$] A subset $S$ of a $\mathbf{K}$-vector space $H$ is \emph{linearly independent} if $0\notin S$ and for all $n\in\N$ and all pairwise distinct $s_1,\ldots,s_n\in S$ we have that $0\notin s_1\boxplus\ldots\boxplus s_n$.\par
\item[$(iii)$] A subset $S$ of a $\mathbf{K}$-vector space $H$ \emph{spans (or generates)} $H$ if for all $x\in H\setminus S$ there exists $n\in\N$ and $s_1,\ldots,s_n\in S$ such that $x\in s_1\boxplus\ldots\boxplus s_n$.
\item[$(iv)$] A \emph{basis} $B$ for a $\mathbf{K}$-vector space $H$ is a linearly independent subset of $H$ which spans~$H$.
\item[$(v)$] A $\mathbf{K}$-vector space $H$ has \emph{dimension} $n\in\N$ if there exists a basis $B$ of $H$ with precisely $n$ elements.
\end{itemize}
\end{definition}

\begin{remark}
As for vector spaces over fields, if a $\mathbf{K}$-vector space $H$ have a basis of cardinality $n\in\N$, then all bases of $H$ have cardinality $n$, this follows from the more general results obtained by J.\ Mittas on vector spaces over hyperfields (see \cite{Mit75}).  
\end{remark}

\begin{lemma}\label{Kdim}
Let $A$ be an integral domain with $1$ and $k$ a proper subfield of $A$. Assume that $A$ is also a $k$-vector space of dimension $n\in\N_{>1}$. Then the additive hypergroup of $A_{k^\times}$ is a $\mathbf{K}$-vector space of dimension $n$.    
\end{lemma}

\begin{proof}
Let $\beta:=\{b_1,\ldots,b_n\}$ be a basis for $A$ over $k$. By Lemma \ref{x+x=0x} it follows that $A_{k^\times}$ is a $\mathbf{K}$-vector space. We first claim that $B:=\{b_1k^\times,\ldots,b_nk^\times\}$ is a basis for $A_{k^\times}$. Indeed, $0k^\times\notin B$ and
\[
0k^\times\in b_{i_1}k^\times\boxplus\ldots\boxplus b_{i_m}k^\times,
\]
for some $1<m\leq n$ and pairwise distinct $1\leq i_1,\ldots,i_m\leq n$ means that there exists $a_{i_1},\ldots,a_{i_m}\in k^\times$ such that
\[
0= a_{i_1}b_{i_1}+\ldots+ a_{i_m}b_{i_m}.
\]
Nevertheless, the latter only happens if $a_{i_1}=\ldots=a_{i_m}=0$ showing that $B$ is a linearly independent subset of $A_{k^\times}$. Let now $xk^\times$ be a nonzero element of $A_{k^\times}\setminus B$. There exists (uniquely determined) $1\leq m\leq n$, $1\leq i_1,\ldots,i_m\leq n$ (pairwise distinct) and $a_{i_1},\ldots,a_{i_m}\in k^\times$ such that
\[
x=a_{i_1}b_{i_1}+\ldots+ a_{i_m}b_{i_m}.
\]
Since $1\in k^\times$, we obtain that
\[
xk^\times=(a_{i_1}b_{i_1}+\ldots+ a_{i_m}b_{i_m})k^\times\in b_{i_1}k^\times\boxplus\ldots\boxplus b_{i_m}k^\times.
\]
Thus, $B$ also spans $A_{k^\times}$. Now it suffices to note that $b_ik^{\times}=b_jk^{\times}$ for $1\leq i,j\leq n$ means that $b_j=ab_i$ for some $a\in k^\times$ and that since $\{b_1,\ldots,b_n\}$ is a basis for $A$ over $k$, we obtain that $b_iK^{\times}=b_jK^{\times}$ holds in $A_{K^\times}$ if and only if $i=j$. Thus, $B$ has cardinality $n$.
\end{proof}

\subsection{The hyperfield of fractions}

Let $R$ be an integral hyperdomain. Define an equivalence relation $\mathfrak{f}$ on $R\times (R\setminus\{0\})$ as
\[
\mathfrak{f}((x,x'),(y,y'))\iff xy'=yx'.
\]
Let $\Fr(R)$ be the corresponding quotient set and denote by $\frac{x}{x'}\in \Fr(R)$ the equivalence class of $(x,x')$ for $x\in R$ and $x'\in R\setminus\{0\}$. Define
\[
\frac{x}{x'}\boxplus \frac{y}{y'}:=\left\{\frac{z}{x'y'}\bigg\vert z\in xy'\boxplus_R yx' \right\}
\]
and
\[
\frac{x}{x'}\cdot \frac{y}{y'}:=\frac{xy}{x'y'}.
\]
for all $\frac{x}{x'},\frac{y}{y'}\in Fr(R)$.

An embedding $\sigma:H\to H'$ of canonical hypergroups is an injective map that satisfies $\sigma(0)=0'$ and
\[
\sigma(x\boxplus y)=(\sigma(x)\boxplus'\sigma(y))\cap\sigma(H).
\]
By an embedding of hyperrings $R\to R'$ we mean an embedding of the additive hypergroups which also preserves the multiplicative structure.
\begin{lemma}\label{fraemb}
For all integral hyperdomains $R$ with unity, $\Fr(R)$ is a hyperfield and $R$ embeds in $\Fr(R)$ via $x\mapsto\frac{x}{1}$.
\end{lemma}

\begin{proof}
The only nontrivial claim is the associativity of the hyperoperation $\boxplus$ of $\Fr(R)$ as defined above. For let $\frac{x}{x'},\frac{y}{y'},\frac{z}{z'}$ be elements of $\Fr(R)$, we have that
\begin{align*}
\left(\frac{x}{x'}\boxplus\frac{y}{y'}\right)\boxplus\frac{z}{z'}&=\bigcup_{t\in xy'\boxplus_R yx'}\frac{t}{x'y'}\boxplus\frac{z}{z'}\\
&=\bigcup_{t\in xy'\boxplus_R yx'}\left\{\frac{s}{x'y'z'}\bigg\vert s\in tz'\boxplus_R zx'y'\right\}\\
&= \left\{\frac{s}{x'y'z'}\bigg\vert s\in \bigcup_{t\in xy'\boxplus_R yx'}tz'\boxplus_R zx'y'\right\}\\
&=\left\{\frac{s}{x'y'z'}\bigg\vert s\in \bigcup_{p\in yz'\boxplus_R zy'}xy'z'\boxplus_R px'\right\}\\
&=\bigcup_{p\in yz'\boxplus_R zy'}\left\{\frac{s}{x'y'z'}\bigg\vert s\in xy'z'\boxplus_R px'\right\}\\
&=\bigcup_{p\in yz'\boxplus_R zy'}\frac{x}{x'}\boxplus\frac{p}{y'z'}\\
&=\frac{x}{x'}\boxplus\left(\frac{y}{y'}\boxplus\frac{z}{z'}\right).
\end{align*}
Where we used associativity and distributivity in $R$ to obtain that:
\begin{align*}
\bigcup_{t\in xy'\boxplus_R yx'}tz'\boxplus_R zx'y'&=(xy'\boxplus_R yx')z'\boxplus_R zx'y'\\
&=xy'z'\boxplus_R (yx'z'\boxplus_R zx'y')\\
&=\bigcup_{p\in yz'\boxplus_R zy'}xy'z'\boxplus_R px'.
\end{align*}
\end{proof}

In the next proposition, we show that taking the (hyper)field of fractions commutes with Krasner's factor construction.

\begin{proposition}\label{Krafrac}
Let $A$ be an integral domain with $1$, $T$ a subgroup of $(A,\cdot,1)$ and $K$ the field of fractions of $A$. Then $\Fr(A_T)$ is isomorphic to $K_T$.
\end{proposition}

\begin{proof}
Since $T$ is a subgroup of $(A,\cdot,1)$ we have that $T=\{\frac{t}{t'}\mid t\in T\}$. It follows that
\begin{align*}
    \sigma:\Fr(A_T)&\to K_T\\
    \frac{xT}{x'T}\mapsto \frac{x}{x'}T
\end{align*}
is well-defined and injective. It is also clearly a surjective map and an isomorphism of the underlying multiplicative groups. It remains to check its behavior on the hyperadditions. On the one hand,
\[
\frac{zT}{z'T}\in\frac{xT}{x'T}\boxplus \frac{yT}{y'T}
\]
holds if and only if $z'T=x'y'T$ and
\begin{equation}\label{111}
zT\in y'xT\boxplus x'yT\iff \exists t,s\in T: z=y'xt+x'ys.
\end{equation}
On the other hand,
\[
\frac{z}{z'}T\in \frac{x}{x'}T\boxplus \frac{y}{y'}T
\]
holds if and only if 
\begin{equation}\label{222}
\exists t,s\in T: \frac{z}{z'}=\frac{x}{x'}t+ \frac{y}{y'}s=\frac{xt}{x'}+ \frac{ys}{y'}=\frac{y'xt+x'ys}{x'y'}.
\end{equation}
By comparing \eqref{111} and \eqref{222} we deduce the desired conclusion.
\end{proof}

\section{Projective spaces, polynomials and rational functions}
\label{Sec3}

In this section we provide an explicit description of the $\mathbf{K}$-vector spaces associated to the projective geometries realised as $\mathbb{P}_F^n$ for some $n\in\N$ and $\mathbb{P}_F^\omega$, for any field $F\neq\mathbb{F}_2$. At this point, our description will be purely algebraic and we will not consider any topology on the projective spaces.

Let us begin with the infinite-dimensional case.

\begin{proposition}\label{infinitepr}
Take a field $F\neq\mathbb{F}_2$ and let $(\mathcal{P}_F,\mathcal{L}_F)$ be the projective geometry realised as $\mathbb{P}^\omega_F$. Then the canonical hypergroup $H(\mathcal{P}_F,\mathcal{L}_F)$ is isomorphic to the additive hypergroup of the factor hyperring $F[X]_{F^\times}$.
\end{proposition}

\begin{proof}
The projective space $\mathbb{P}^\omega_F$ is given by $F^\omega\setminus\{0\}/\mathfrak{p}$. Since (as $F$-vector spaces) we have a canonical isomorphism $f:F^\omega\to F[X]$, by definition of $\mathfrak{p}$ we obtain a bijection: 
\begin{align*}
\mathbb{P}^\omega_F&\to F[X]/F^\times\setminus\{0\}\\
\hat x&\mapsto f(\bar x)F^\times
\end{align*}
which may be extended to a bijection
\[
\sigma:R:=H(\mathcal{P}_F,\mathcal{L}_F)\to F[X]_{F^\times}=:R'. 
\]
We claim that $\sigma$ is an isomorphism of canonical hypergroups, i.e., $\sigma(0)=0'$ and $\sigma(x\boxplus y)=\sigma(x)\boxplus'\sigma(y)$ for all $x,y\in R$. Let $x,y\in R$ be nonzero. Then $z\in x\boxplus y$ if and only if $\hat z\in\ell(\hat x,\hat y)\setminus\{\hat x,\hat y\}$ holds in $\mathbb{P}^\omega_F$ for the corresponding projective points. By definition of $\ell(\hat x,\hat y)$, for any choice of representatives $\bar x,\bar y,\bar z$ of $\hat x,\hat y, \hat z$, respectively, we may find $a,b\in F^\times$ such that $\bar z=a\bar x+b\bar y$, which via the isomorphism $f$, is equivalent to $f(\bar z)=af(\bar x)+bf(\bar y)$. Meaning that $\sigma(z)=f(\bar z)F^\times\in f(\bar x)F^\times\boxplus f(\bar y)F^\times=\sigma(x)\boxplus'\sigma(y)$.
\end{proof}

\begin{remark}
By pulling back the multiplication of $F[X]_{F^\times}$ via the isomorphism of the previous theorem we find a natural monoid structure on infinite (countable) dimensional projective spaces over fields. 
\end{remark}

We deduce the following as a corollary of Propositions \ref{infinitepr}, \ref{Krafrac} and Lemma \ref{fraemb} above.

\begin{corollary}\label{F(X)F}
Take a field $F\neq\mathbb{F}_2$ and let $(\mathcal{P}_F,\mathcal{L}_F)$ be the projective geometry realised as $\mathbb{P}^\omega_F$. The canonical hypergroup $H(\mathcal{P}_F,\mathcal{L}_F)$ embeds into the additive hypergroup of the factor hyperfield of the field of rational functions $F(X)$ modulo $F^\times$.
\end{corollary}

We can now pass to the finite-dimensional case.

\begin{proposition}
Take a field $F\neq\mathbb{F}_2$ and let $(\mathcal{P}^n_F,\mathcal{L}^n_F)$ be the projective geometry realised as $\mathbb{P}^n_F$ for some $n\in\mathbb{N}$. The canonical hypergroup $H(\mathcal{P}^n_F,\mathcal{L}^n_F)$ embeds into the additive hypergroup of the factor hyperfield $F(X)_{F(X^{n+1})^\times}$.
\end{proposition}

\begin{proof}
Let $F[X]_n$ be the $F$-vector space of polynomials of degree at most $n$.  From the canonical isomorphism of $F$-vector spaces $F^{n+1}\simeq F[X]_n$, we obtain an embedding of canonical hypergroups:
\[
\sigma:H(\mathcal{P}^n_F,\mathcal{L}^n_F)\to F[X]_{F^\times}\hookrightarrow F(X)_{F^\times}.
\]
Now, for all $x\in H(\mathcal{P}^n_F,\mathcal{L}^n_F)$ we have that $\sigma(x)=p(X)F^\times$ for some polynomial $p(X)\in F[X]$ of degree at most $n$. It follows that the assignment $p(X)F^\times\mapsto p(X)F(X^{n+1})^\times$ gives the required embedding.
\end{proof}

If $A$ is an $F$-algebra of dimension $n\in\N_{>1}$ with $1$ and no nonzero zero-divisors, then $F$ is canonically identified with the subset $1\cdot F$ of $A$. Up to this identification, $A_{F^\times}$ is a $\mathbf{K}$-vector space and if $A'$ is another $F$-algebra of dimension $n$ with $1$ and no nonzero zero-divisors, then the isomorphism $A\simeq A'$ as $F$-vector spaces induces an isomorphism of $\mathbf{K}$-vector spaces $A_{F^\times}\simeq A'_{F^\times}$. This proves:

\begin{corollary}\label{PnF}
Take a field $F\neq\mathbb{F}_2$ and let $(\mathcal{P}^n_F,\mathcal{L}^n_F)$ be the projective geometry realised as $\mathbb{P}^n_F$ for some $n\in\mathbb{N}$. Let $A$ be an $F$-algebra with $1$ and no nonzero zero-divisors, with dimension $n+1$. Then the canonical hypergroup $H(\mathcal{P}_F,\mathcal{L}_F)$ is isomorphic to the additive hypergroup of the factor hyperring of $A_{F^\times}$. 
\end{corollary}

\begin{corollary}\label{finprstr}
Take a field $F\neq\mathbb{F}_2$ and let $(\mathcal{P}^n_F,\mathcal{L}^n_F)$ be the projective geometry realised as $\mathbb{P}^n_F$ for some $n\in\mathbb{N}$. If $K$ is a field extension of $F$ of degree $n+1$, then the canonical hypergroup $H(\mathcal{P}^n_F,\mathcal{L}^n_F)$ is isomorphic to the additive hypergroup of the factor hyperfield $K_{F^\times}$. In particular, If $p$ is an irreducible polynomial of degree $n+1$ over $F$, then the canonical hypergroup $H(\mathcal{P}^n_F,\mathcal{L}^n_F)$ is isomorphic to the additive hypergroup of the factor hyperfield of $F[X]/(p)$ modulo $F^\times$.
\end{corollary}

\begin{corollary}
Let $F\neq\mathbb{F}_2$ be a field and $(\mathcal{P},\mathcal{L})$ be the projective geometry realised as $\mathbb{P}^n_F$ for some $n\in\N$. The dimension of $H(\mathcal{P},\mathcal{L})$ as a $\mathbf{K}$-vector space is $n+1$.
\end{corollary}

\begin{proof}
By Lemma \ref{Kdim}, the dimension of $H(\mathcal{P},\mathcal{L})$ as a $\mathbf{K}$-vector space is equal to the dimension of $F[X]/( X^{n+1})$ as an $F$-vector space, i.e., $n+1$. 
\end{proof}

\section{Continuous incidence abelian group structures}
\label{Sec4}

The algebraic description of the $\mathbf{K}$-vector spaces associated to projective spaces shows that these also naturally admit a multiplicative structure. In this section, under the assumption that the base field is a local field, we will study the situations in which this multiplicative structure preserves the topological and geometric structure of the corresponding finite-dimensional projective spaces. We will mainly focus on the cases in which the multiplicative structure is that of an abelian group.

\begin{remark}
Since in Corollaries \ref{PnF} and \ref{finprstr} all $F$-vector spaces are finite-dimensional, if $F$ is a local field, then all the isomorphisms as $F$-vector spaces, and thus of the corresponding $\mathbf{K}$-vector spaces, are also homeomorphisms (cf.\ Remark \ref{Conrad}). This means that we can assume all finite-dimensional projective spaces over $F$ to be of the form $\mathbb{P}_F^n$ for some $n\in\N$.
\end{remark}

\begin{definition}
Let $(\mathcal{P},\mathcal{L})$ and $(\mathcal{P}',\mathcal{L}')$ be projective geometries. A pair of bijections 
\[
\alpha:\mathcal{P}\to\mathcal{P}',\quad \beta:\mathcal{L}\to\mathcal{L}'
\]
is called a \emph{collineation} between $(\mathcal{P},\mathcal{L})$ and $(\mathcal{P}',\mathcal{L}')$ if $\ell(\alpha(x),\alpha(y))=\beta(\ell(x,y))$ for all distinct $x,y\in\mathcal{P}$.\par
The set of all collineations between a projective geometry $(\mathcal{P},\mathcal{L})$ and itself is a group with respect to composition. This is called the \emph{collineation group} of $(\mathcal{P},\mathcal{L})$ and we will denote it by $\Col(\mathcal{P},\mathcal{L})$.
\end{definition}
The following is a straightforward observation.
\begin{lemma}\label{azaxay}
A bijection $\alpha:\mathcal{P}\to\mathcal{P}$ is (part of) a collineation if and only if it preserves the incidence relation, i.e., for all distinct $a,b\in\mathcal{P}$
\begin{equation}\label{collalph}
x\in\ell(a,b)\iff\alpha(x)\in\ell(\alpha(a),\alpha(b)).
\end{equation}
\end{lemma}

\begin{proof}
Assume that $(\alpha,\beta)$ is a collineation. Take $a,b\in\mathcal{P}$ and $x\in\ell(a,b)$. If $\alpha(x)\notin\ell(\alpha(a),\alpha(b))$, then 
\[
\ell(\alpha(x),\alpha(a))\neq \ell(\alpha(a),\alpha(b))=\beta(\ell(a,b))=\beta(\ell(x,a))=\ell(\alpha(x),\alpha(a)),
\]
a contradiction. For the other implication in \eqref{collalph}, we can just replace in the above reasoning $(\alpha,\beta)$ with $(\alpha^{-1},\beta^{-1})$.\par
Now assume that \eqref{collalph} holds. We have to show that $\beta(\ell(a,b)):=\ell(\alpha(a),\alpha(b))$ for distinct $a,b\in\mathcal{P}$, defines a bijection $\beta:\mathcal{L}\to\mathcal{L}'$. Since $\alpha$ is bijective, $\beta$ is surjective. For the injectivity, take $\ell(a,b),\ell(x,y)\in\mathcal{L}$ with $\beta(\ell)=\beta(\ell')$. Then
\[
\alpha(a),\alpha(b)\in \ell(\alpha(a),\alpha(b))=\ell(\alpha(x),\alpha(y))
\]
implies by \eqref{collalph} that $a,b\in\ell(x,y)$ and $\ell(a,b)=\ell(x,y)$ follows.
\end{proof}

\begin{remark}\label{coltriv}
    Note that if $\mathcal{L}$ is a singleton (i.e., the projective geometry consists of a unique line), then any bijection $\mathcal{P}\to\mathcal{P}$ is a collineation.
\end{remark}

\begin{definition}
Assume that two projective geometries $(\mathcal{P},\mathcal{L})$ and $(\mathcal{P}',\mathcal{L}')$ are given endowed with a topology defined on their sets of points. A collineation $(\alpha,\beta)$ between $(\mathcal{P},\mathcal{L})$ and $(\mathcal{P}',\mathcal{L}')$ will be called \emph{continuous} if $\alpha:\mathcal{P}\to\mathcal{P}'$ is continuous with respect to the given topologies.
\end{definition}

We have now come to the definition of one of the main concepts in our investigation.

\begin{definition}
We say that a projective geometry $(\mathcal{P},\mathcal{L})$ realised as a projective space over a local field $F$ \emph{admits a (continuous) incidence abelian group structure} if $\mathcal{P}$ is an abelian group with respect to some operation $\cdot$ such that $(\alpha_a,\beta_a)$ is a (continuous) collineation between $(\mathcal{P},\mathcal{L})$ and itself (with respect to the standard topology of the projective space, induced by the canonical topology on $F$), for all $a\in\mathcal{P}$, where $\alpha_a:\mathcal{P}\to\mathcal{P}$ and $\beta_a:\mathcal{L}\to\mathcal{L}$ are defined, for all $a,x,y\in\mathcal{P}$, with $x\neq y$, as $\alpha_a(x):=a\cdot x$ and $\beta(\ell(x,y)):=\ell(a\cdot x,a\cdot y)$. 
\end{definition}

\begin{lemma}\label{Ellis}
 A projective geometry $(\mathcal{P},\mathcal{L})$ realised as a finite-dimensional projective space over a local field $F$ admits a continuous incidence abelian group structure if and only if the abelian group $(\mathcal{P},\cdot)$ is a topological group with respect to the standard topology on $\mathcal{P}$ and $a\cdot z\in\ell(a\cdot x,a\cdot y)$ for all $a\in\mathcal{P}$, distinct $x,y\in\mathcal{P}$ and $z\in\ell(x,y)$.    
\end{lemma}

\begin{proof}
Since projective spaces over local fields are compact and thus in particular locally compact spaces, the result follows from Lemma \ref{azaxay} and the fundamental papers of R.\ Ellis \cite{Ell57i,Ell57}, where it is proved that locally compact semitopological groups are topological.
\end{proof}

\begin{lemma}\label{hyptwwo}
A projective geometry $(\mathcal{P},\mathcal{L})$ admits an incidence abelian group structure if and only if $H(\mathcal{P},\mathcal{L})$ is (isomorphic to) the additive hypergroup of some hyperfield.    
\end{lemma}

\begin{proof}
If $H(\mathcal{P},\mathcal{L})$ is (up to isomorphism) the additive hypergroup of some hyperfield, then the multiplicative abelian group of the hyperfield is an incidence abelian group structure on $\mathcal{P}$, since $\alpha_a^{-1}=\alpha_{a^{-1}}$ for all $a\in\mathcal{P}$ by distributivity (axiom (HR3) of Definition \ref{Krasnerhyp}).\par
Conversely, after setting $0\cdot x=x\cdot 0=0$ for all $x\in H(\mathcal{P},\mathcal{L})$, where $\cdot$ is an incidence abelian group structure on $\mathcal{P}$, it suffices to verify the distributivity property:
\[
z\in x\boxplus y=\ell(x,y)\setminus\{x,y\}
\]
if and only if
\[
az\in ax\boxplus ay=\ell(ax,ay)\setminus\{ax,ay\}
\]
for all $a,x,y,z\in\mathcal{P}$, with $x\neq y$.
But $z\notin\ell(x,y)\setminus\{x,y\}$ means that $z\notin\ell(x,y)$ or $z\in\{x,y\}$. By Lemma \ref{azaxay} and the bijectivity of $\alpha_a$, this happens if and only if $az\notin\ell(ax,ay)$ or $az\in\{ax,ay\}$ and that is, $az\notin\ell(ax,ay)\setminus\{ax,ay\}$.
\end{proof}

The validity of the following proposition is not immediate to verify but will be useful.

\begin{proposition}\label{tophyp}
Let $H\simeq K_{k^\times}$ be a non-discrete, compact and Hausdorff topological hyperfield obtained as the factor of a field extension $K|k$ of degree $d\in\N_{>1}$. Then $k$ is a local field and $H^\times$ is isomorphic as a topological group to $K^\times/k^\times$ with respect to the quotient topology, where $K\simeq k^d$ is endowed with the product topology.      
\end{proposition}

\begin{proof}
Since $k$ is a subfield of $K$ we have that $H$ is a $\mathbf{K}$-vector space by Lemma \ref{x+x=0x}. The projective geometry $\mathcal{P}$ associated to $H$ is finite-dimensional and clearly collinear to the projective geometry realised as $\mathbb{P}_k^{d-1}$. 

By assumption $H^\times$ is a compact and Hausdorff topological abelian group isomorphic (as a group) to the quotient group $K^\times/k^\times$. Let $\varphi:K^\times/k^\times\to H^\times$ be an isomorphism and define $O\subseteq K^\times/k^\times$ to be open if and only if $\varphi(O)$ is open in $H^\times$. In this way we obtain a topology $\tau$ on $K^\times/k^\times=\mathbb{P}_k^{d-1}$ with respect to which we have a compact and Hausdorff topological abelian group homeomorphic via $\varphi$ to $H^\times$. 

For all distinct $x,y\in H^\times$ we have that $0\notin x\boxplus y\cup\{x,y\}=:S$ and hence $S\subseteq H^\times$. As a subset of $\mathbb{P}_k^{d-1}$, the set $S$ is a projective subspace of dimension $1$ (by the converse of Proposition \ref{CC}), i.e., $S\simeq\mathbb{P}_k^1$ (as a projective space) and the subspace topology induced on $S$ from $H^\times$ (via $\varphi$) defines a topology $\tau'$ on $\mathbb{P}_k^1$. Since $H^\times$ is Hausdorff, it follows that $S$ is closed in $H^\times$ and thus $(\mathbb{P}_k^1,\tau')$ is compact and Hausdorff. Therefore, $k$ with the induced subspace topology is locally compact. We obtained a locally compact topology on $k$.

The additive and multiplicative group laws of $k$ are continuous with respect to this topology, because they coincide with the additive and multiplicative group laws of $K$ by assumption.

Since $\mathcal{P}$ and $\mathbb{P}_k^{d-1}$ are projective geometries of the same dimension, we deduce that $\mathcal{P}$ is isomorphic to $\mathbb{P}_k^{d-1}$ as a projective space and thus homeomorphic to it, when the latter is endowed with its standard topology coming from the topology of the local field $k$ (see Remark \ref{Conrad}). This completes the proof.
\end{proof}

\begin{corollary}\label{hyptwo}
A projective geometry $(\mathcal{P},\mathcal{L})$ realised as a projective space of dimension $n\in\N$ over a local field $F$ admits a continuous incidence abelian group structure if and only if $H(\mathcal{P},\mathcal{L})$ is isomorphic (and homeomorphic) to the additive hypergroup of some topological hyperfield which is a factor hyperfield of the form $L_{K^\times}$ for some field extension $L|K$ of degree $n+1$, where $K$ is a local field, $L\simeq K^{n+1}$ (where the isomorphism is of $K$-vector spaces) is endowed with the product topology and $L^\times/K^\times$ with the quotient topology.
\end{corollary}

\begin{proof}
    By \cite[Theorem 3.8]{CC11}, since the projective geometry realised as $\mathbb{P}_F^n$ is Desarguesian, we have that $H\simeq L_{K^\times}$ as $\mathbf{K}$-vector spaces for some proper field extension $L|K$. The result then follows by Proposition \ref{tophyp} above.
\end{proof}

\section{The case of projective lines}
\label{Sec5}

We observe that in dimension $1$, the geometry (i.e., the incidence relation) is somehow trivial and thus the only restrictions on continuous incidence abelian group structures come from the topology.

In the Archimedean case, the only projective line which admits a continuous incidence abelian group structure is the real projective line and this structure is uniquely determined up to isomorphism of topological groups. This follows by the following facts as it is well-known that $\mathbb{P}_\mathbb{R}^1\approx\mathbb{S}^1$.

\begin{fact}[\cite{VN33}, or Theorem 9.57 in \cite{Hofmann20}]\label{VonNeumann}
Compact locally Euclidean topological groups are Lie groups.
\end{fact}

\begin{fact}[Theorem 6.95 in \cite{Hofmann20}]
\label{thm:695}
The only compact Lie groups on a sphere are $\mathbb{S}^1$ and $\mathbb{S}^3$, with their standard topological group structure.   
\end{fact}

We now pass to the non-Archimedean case, which is more subtle, due to the following observation.

\begin{fact}
Any two non-Archimedean local fields are homeomorphic.   
\end{fact}

\begin{proof}
   Their valuation rings are homeomorphic to a compact, perfect and totally disconnected subset of $\mathbb{R}$ and hence all homeomorphic to the Cantor set (see e.g.\ \cite[Theorem 2.29]{Kat07}). Any local field, being the field of fractions of its valuation ring, is then homeomorphic to the Cantor set without a point.
\end{proof}

Since the projective line over a topological field $(K,\tau)$ is homeomorphic to the Alexandroff extension of $(K,\tau)$, we deduce from the above proposition that $\mathbb{P}_{F}^1\approx \mathbb{P}_{F'}^1$, for any pair of non-Archimedean local fields $F$ and $F'$.

If $F$ is a non-Archimedean local field, then by the Artin-Schreier Theorem \cite{AS26,AS27} and since $F$ is not algebraically closed nor real closed\footnote{A real closed field must have characteristic $0$, and a finite extension $F$ of $\mathbb{Q}_p$ with degree $d\in\N$ is not real closed because it is not elementarily equivalent to $\mathbb{R}$ as e.g.\ $\exists x:x^{2d}-p$ holds in $\mathbb{R}$ but not in $F$.}, the algebraic closure of $F$ has infinite degree over $F$. In particular, there exists an extension $K$ of $F$ of degree $2$ and the topological group $K^\times/F^\times$ gives a continuous incidence abelian group structure on $\mathbb{P}_F^1$. By the above observations, the latter is also a continuous incidence abelian group structure on $\mathbb{P}_{F'}^1$ for any non-Archimedean local field $F'$. 

Clearly, all extensions of degree $2$ of $F$ are isomorphic (as fields) and thus yield isomorphic continuous incidence abelian group structures. However, if $F$ and $F'$ are non-isomorphic local fields, then the corresponding continuous incidence abelian group structures on $\mathbb{P}_{F}^1\approx\mathbb{P}_{F'}^1$ will not be isomorphic as groups. 

We end up with infinitely many non-isomorphic continuous incidence abelian group structures on $\mathbb{P}_F^1$, one for each other non-Archimedean local field $F'$. Note that by the classification of (non-Archimedean) local fields, these form a countable collection. By Corollary \ref{hyptwo}, there are no other possibilities.

\section{The case of dimension \texorpdfstring{$n>1$}{TEXT}}
\label{Sec6}

In this final section we will prove that over Archimedean local fields, finite-dimensional projective spaces of dimension at least $2$ do not admit incidence abelian group structures and in particular continuous incidence abelian group structures. 

In contrast, we will also provide an explicit construction of continuous incidence abelian group structures for the non-Archimedean case in any (finite) dimension $n>1$. Moreover, we will show that there are at most countably many of these and finitely many if $n+1$ is not divisible by the characteristic of the base non-Archimedean local field.

The key point is that the geometry (i.e., the incidence relation) in dimension greater than $1$ does not trivialise as in the case of projective lines and, what is more, is able to provide enough information to recover the full field structure of the base local field.

For a field $F$, we will denote by $\Aut(F)$ the group (under composition) containing all the automorphisms of $F$. If $(F,\tau)$ is a topological field, then $\Aut_\tau(F)$ will denote the subgroup of $\Aut(F)$ whose elements are continuous with respect to $\tau$ (recall that all the elements of $\Aut_\tau(F)$ are automatically homeomorphisms).

\begin{remark}
Since $[\mathbb{C}:\mathbb{R}]=2$, there are only $2$ automorphisms of $\mathbb{C}$ that fix $\mathbb{R}$. However, as explained e.g.\ in \cite[Example 19.20]{Mor96}, one can deduce (using the axiom of choice) that there are infinitely many $\theta\in\Aut(\mathbb{C})$ such that $\theta(\mathbb{R})\neq\mathbb{R}$ (for more information see e.g.\ \cite{Kest51, Yale66}). Since $\mathbb{Q}$ is fixed by any $\theta\in\Aut(\mathbb{C})$ and $\mathbb{Q}$ is dense in $\mathbb{R}$, it follows all these wild automorphisms of $\mathbb{C}$ are not in $\Aut_\tau(\mathbb{C})$, which contains only the identity and complex conjugation.     
\end{remark}

\begin{definition}
Let $F$ be a field and $V,W$ vector spaces over $F$. A \emph{semilinear map from $V$ to $W$} is a map $f:V\to W$ which is linear \lq\lq up to a (fixed) automorphism of $F$\rq\rq, that is there exists $\theta_f\in\Aut(F)$ such that 
\begin{itemize}[leftmargin=2cm]
    \item[(SL1)] $f(v+v')=f(v)+f(v')$, for all $v,v'\in V$.
    \item[(SL2)] $f(av)=\theta_f(a)f(v)$, for all $v\in V$.
\end{itemize}
\end{definition}

The set $\GammaL(V,F)$ of all bijective semilinear maps $V\to V$, where $V$ is some $F$-vector space, is called the \emph{general semilinear group} of $V$. The general linear group $\GL(V,F)$ of $V$ over $F$ is a normal subgroup of $\GammaL(V,F)$ and the quotient group $\GammaL(V,F)/\GL(V,F)$ is isomorphic to $\Aut(F)$ via $f\GL(V,F)\mapsto\theta_f$. This means that the general semilinear group of $V$ over $F$ is (isomorphic to) the semidirect product of the general linear group of $V$ over $F$ and the automorphism group of $F$:
\[
\GammaL(V,F)\simeq \GL(V,F)\rtimes\Aut(F).
\]

\begin{definition}
The \emph{projective general semilinear group of an $F$-vector space} $V$ is
\[
\PGammaL(V,F):=\PGL(V,F)\rtimes\Aut(F).
\]
We will also write $\PGammaL(n,F)$ for $\PGammaL(F^{n+1},F)$.
\end{definition}

\begin{lemma}\label{contsemil}
If $(F,\tau)$ is a topological field and $n,m\in\N$, then a semilinear map $f:F^n\to F^m$ is continuous (with respect to the product topologies) if and only if $\theta_f\in\Aut_\tau(F)$.
\end{lemma}

\begin{proof}
It suffices to prove the statement for non-trivial semilinear maps $f:F^n\to F$. Let $v_1\in F^n$ be such that $f(v_1)=1$. Complete $v_1$ to a basis $v_1,\ldots,v_n$ of $F^n$. After eventually replacing $v_i$ with $f(v_i)^{-1}v_i$, we can assume that $f(v_i)=1$ for all $1\leq i\leq n$. By semilinearity we obtain that
\[
f\left(\sum_{i=1}^n a_iv_i\right)=\sum_{i=1}^n\theta_f(a_i).
\]
Therefore, if $\theta_f$ is continuous, then $f$ is continuous. Conversely, if $f$ is continuous, then for any fixed $a\in F$ we have that
\[
\theta_f(a)=\theta_f(a)f(v_1)=f(av_1)
\]
and thus $\theta_f$ is continuous at $a$, since scalar multiplication by $a$ and $f$ are continuous. Since $a$ was arbitrary, this proves the continuity of $\theta_f$.
\end{proof}

Motivated by the above observation we define:

\begin{definition}
Let $(F,\tau)$ be a topological field. The \emph{projective general continuous semilinear group of an $F$-vector space} $V$ is
\[
\PGammaL_\tau(V,F):=\PGL(V,F)\rtimes\Aut_\tau(F).
\]
We will also write $\PGammaL_\tau(n,F)$ for $\PGammaL_\tau(F^{n+1},F)$. 
\end{definition}

We now shift our attention to the description of collineations between projective geometries as isomorphisms of the associated $\mathbf{K}$-vector spaces.

\begin{lemma}\label{isocoll}
Let $(\mathcal{P},\mathcal{L})$ and $(\mathcal{P}',\mathcal{L}')$ be projective geometries. Then $H(\mathcal{P},\mathcal{L})$ and $H(\mathcal{P}',\mathcal{L}')$ are isomorphic as $\mathbf{K}$-vector spaces if and only if there is a collineation between $(\mathcal{P},\mathcal{L})$ and $(\mathcal{P}',\mathcal{L}')$.   
\end{lemma}

\begin{proof}
This easily follows from Lemma \ref{azaxay}.
\end{proof}

\begin{lemma}
Let $F,F'$ be fields and $V,W$ vector spaces over $F$ and $F'$, respectively. Take a bijective semilinear map $f:V\to W$ and define
\begin{align*}
\alpha_f:\mathbb{P}(V)&\to\mathbb{P}(W)\\
[v]&\mapsto[f(v)]
\end{align*}
Then $\alpha_f$ and $\beta_f(\ell([v],[v'])):=\ell(\alpha([v]),\alpha([v']))$ ($v,v'\in \mathbb{P}(V)$, $v\neq v'$) are bijections. In particular, $(\alpha_f,\beta_f)$ is a collineation between the projective geometries $(\mathcal{P},\mathcal{L})$ and $(\mathcal{P}',\mathcal{L}')$ realised as $\mathbb{P}(V)$ and $\mathbb{P}(W)$, respectively. If $F$ and $F'$ are topological fields, $V$ and $W$ are finite-dimensional and $f$ is continuous (with respect to the product topologies), then the collineation $(\alpha_f,\beta_f)$ is continuous (with respect to the quotient topologies).
\end{lemma}

\begin{proof}
The bijectivity of $\alpha_f$ follows easily from the bijectivity of the semilinear map $f$ and the surjectivity of $\alpha_f$ implies the surjectivity of $\beta_f$. It remains to verify that $\beta_f$ is injective. If $\ell:=\ell([v],[v'])$ and $\ell'=\ell([u],[u'])$ are lines in $\mathcal{L}$ such that $\beta_f(\ell)=\beta_f(\ell')$. By definition of $\alpha_f$ and $\beta_f$, we obtain that $\ell([f(v)],[f(v')])=\ell([f(u)],[f(u')])$. In particular,
\[
[f(u)],[f(u')]\in\ell([f(v)],[f(v')]),
\]
so 
\begin{align*}
f(u)&=af(v)+bf(v')=f(\theta_f^{-1}(a)v+\theta_f^{-1}(b)v')~~\text{ and}\\
f(u')&=cf(v)+df(v')=f(\theta_f^{-1}(c)u+\theta_f^{-1}(d)u')
\end{align*}
holds for some $a,b,c,d\in F^\times$, where we have used (SL1) and (SL2). The bijectivity of $f$ now implies that $u=\theta_f^{-1}(a)v+\theta_f^{-1}(b)v'\in\ell$ and $u'=\theta_f^{-1}(c)u+\theta_f^{-1}(d)u'\in\ell$. Thus, $\ell=\ell'$ and therefore $\beta_f$ is injective.

The last assertion on continuity is straightforward to verify on the basis of Lemma \ref{contsemil}. 
\end{proof}

\begin{fact}[Theorem 3.5.8 in \cite{BR98}]\label{FTPr}
Let $F_1,F_2$ be fields and $n,m\in\N_{>1}$. If there is a collineation $(\alpha,\beta)$ between the projective geometries $(\mathcal{P},\mathcal{L})$ and $(\mathcal{P}',\mathcal{L}')$ realised as $\mathbb{P}_{F_1}^n$ and $\mathbb{P}_{F_2}^m$, respectively, then $n=m$ and there is a unique bijective semilinear map $f:F_1^{n+1}\to F_2^{n+1}$ such that $\alpha=\alpha_f$ and $\beta=\beta_f$. In particular, $F_1$ and $F_2$ are isomorphic as fields.
\end{fact}

\begin{remark}
It follows by the definition of the quotient topology that if the collineation given by hypothesis of the above theorem is continuous, then the semilinear map which induces it must be continuous as well.    
\end{remark}

\begin{corollary}
Let $n>1$ be an integer. If $F$ is any field, then collineation group of the projective geometry realised as the projective space $\mathbb{P}_F^n$ is isomorphic to the projective general semilinear group $\PGammaL(n,F)$. If $(F,\tau)$ is a Hausdorff topological field, then the group\footnote{Recall that a continuous bijection between compact Hausdorff spaces is a homeomorphism.} of continuous collineations of the projective geometry realised as the projective space $\mathbb{P}_F^n$ is isomorphic to the projective general continuous semilinear group $\PGammaL_\tau(n,F)$.
\end{corollary}

\begin{theorem}\label{keythm}
Take a field $F\neq\mathbb{F}_2$. If the projective geometry associated $\mathbb{P}_F^n$ for some integer $n>1$ admits an incidence abelian group structure, then the $\mathbf{K}$-vector space associated to the corresponding projective geometry is isomorphic to a $\mathbf{K}$-vector space of the form $K_{k^\times}$ for some proper finite field extension $K|k$ of degree $n+1$. In addition, $k$ and $F$ are isomorphic as fields.
\end{theorem}

\begin{proof}
If the projective geometry associated $\mathbb{P}_F^n$ for some $n\in\N_{>1}$ admits an incidence abelian group structure, then the associated $\mathbf{K}$-vector space $\mathbb{H}_F^n$, is the additive hypergroup of some hyperfield $H$, by Lemma \ref{hyptwwo}. By \cite[Theorem 3.8]{CC11}, since the projective geometry realised as $\mathbb{P}_F^n$ is Desarguesian (see e.g.\ \cite[Theorem 2.2.1]{BR98}), we have that $H\simeq K_{k^\times}$ as $\mathbf{K}$-vector spaces for some proper field extension $K|k$. Let
\[
\sigma:\mathbb{H}_F^n\to K_{k^\times}
\]
be this isomorphism. Since an isomorphism of $\mathbf{K}$-vector spaces sends a basis to a basis, we obtain that the dimension of $K_{k^\times}$ as a $\mathbf{K}$-vector space must be $n+1$. By Lemma \ref{Kdim}, we conclude that $[K:k]=n+1$. In addition, $\sigma$ induces a collineation $(\alpha,\beta)$ between $\mathbb{P}_F^{n}$ and the projective geometry associated to the additive hypergroup of $K_{k^\times}$.\par
The $\mathbf{K}$-vector space $\mathbb{H}_k^n$ associated to the projective geometry realised as $\mathbb{P}_k^n$ is up to isomorphism the additive hypergroup of $K_{k^\times}$ by Corollary \ref{PnF}. We have thus obtained a collineation between the projective geometry realised as $\mathbb{P}_F^n$ and the one realised as $\mathbb{P}_k^n$. It follows now from Fact \ref{FTPr} that $F\simeq k$ as fields.
\end{proof}

From the proof of the above theorem and Lemma \ref{isocoll}, we deduce the following result.

\begin{corollary}
Let $F$ be a field and let $n\in\N_{>1}$. The projective geometry $(\mathcal{P},\mathcal{L})$ realised as $\mathbb{P}_F^n$ admits an incidence abelian group structure with operation $\cdot$ and neutral element $e$ if and only if there is a field extension $K$ of $F$ such that $[K:F]=n+1$ and $(H(\mathcal{P},\mathcal{L}),\cdot,e)\simeq K_{F^\times}$ as hyperfields. Moreover, $\PGammaL(n,F)$ is a subgroup of the group of automorphisms of the multiplicative quotient group $K^\times/F^\times$. 
\end{corollary}

\begin{corollary}
Let $(F,\tau)$ be a local field and let $n\in\N_{>1}$. The projective geometry $(\mathcal{P},\mathcal{L})$ realised as $\mathbb{P}_F^n$ admits a continuous incidence abelian group structure with operation $\cdot$ and neutral element $e$ if and only if there is a field extension $K$ of $F$ of degree $n+1$ and such that $(H(\mathcal{P},\mathcal{L}),\cdot,e)\simeq K_{F^\times}$ as topological hyperfields. Moreover, $\PGammaL_\tau(n,F)$ is a subgroup of the group of continuous automorphisms of the topological quotient group $K^\times/F^\times$.
\end{corollary}

\begin{corollary}
Let $F$ be an algebraically closed field or a real closed field and $n\in\N_{>1}$. Then the projective geometry realised as $\mathbb{P}_F^n$ does not admit an incidence abelian group structure.    
\end{corollary}

\begin{proof}
The existence of such a structure would imply the existence of a (finite and thus algebraic) field extension $K$ of $F$ of degree $n+1>2$. If $F$ is algebraically closed, then the only algebraic extension over $F$ has degree $1<n+1$ and if $F$ is real closed, then the algebraic closure of $F$ has degree $2<n+1$ over $F$ by the Artin-Schreier Theorem.
\end{proof}

\begin{corollary}\label{RCniet}
The projective geometries realised as $\mathbb{P}^n_F$, where $F$ is an archimedean local field and $n\in\N_{>1}$, do not admit incidence abelian group structures. In particular, they do not admit continuous incidence abelian group structures.    
\end{corollary}

\begin{theorem}\label{mainth}
Let $F$ be a non-Archimedean local field and $n\in\N_{>1}$. The projective geometry realised as $\mathbb{P}_F^n$ admit countably many (up to isomorphism of topological groups) continuous abelian incidence group structures. If $n+1$ is not divisible by $\Char(F)$, then $\mathbb{P}_F^n$ admits finitely many continuous abelian incidence group structures.
\end{theorem}

\begin{proof}
Fix $n\in\N$. Since the algebraic closure of $F$ has infinite degree over $F$, there exists an irreducible polynomial $p_n\in F[X]$ of degree $n+1$. By Corollary \ref{finprstr}, the canonical hypergroups associated to the projective geometries realised as $\mathbb{P}_F^n$ are isomorphic to the additive hypergroup of the factor hyperfield of $F[X]/(p_n)$ modulo $F^\times$. Let $H$ denote this hyperfield. By Corollary \ref{hyptwo}, if for all $a\in H^\times$, the map $x\mapsto ax$ from $H^\times$ to $H^\times$ is continuous, where $H^\times=\mathbb{P}_F^n$ is endowed with the standard topology of the $n$-dimensional projective space, then these give all the possibilities. Let $O$ be open in $H^\times$. This means that $\pi^{-1}(U)$ is open in $F[X]/(p_n)\simeq F^{n+1}$ and does not contain $0$. It thus suffices to show that
\begin{align*}
f_q:F[X]/(p_n)&\to F[X]/(p_n)\\
p+(p_n)&\mapsto qp+(p_n)
\end{align*}
is continuous with respect to the product topology, for all nonzero $q+(p_n)\in F[X]/(p_n)$. This happens if and only if $\pi_i\circ f_q: F[X]/(p_n)\simeq F^{n+1}\to F$ is continuous for all $1\leq i\leq n+1$, where $\pi_i(\bar v)$ gives the $i$-th component of the vector $\bar v\in F^{n+1}$. Since the $i$-th coefficient of a product of polynomials modulo $p_n$ is computed using only the additive and the multiplicative structure of $F$, which is a topological field, the desired continuity follows. 

The assertions on the cardinality of the collection of all homeo-isomorphism classes of continuous incidence abelian group structures on $\mathbb{P}_F^n$ will be proven if we are able to show that the non-Archimedean local field $F$ admits countably many finite extensions $K$ of degree $n+1$ and finitely many in case $\Char(F)$ does not divide $n+1$ (cf. Corollary \ref{finprstr}). The latter is a well-known fact (see e.g.\ \cite[page 83, Theorem 6 and subsequent remark]{Fal08})
%, since (cf.\ Appendix \ref{app1}) $n+1=e(K|F)f(K|F)$ for all such $K$, by \cite[page 69, Theorem 1]{Fal08}. 
and we note that $\Char(\mathbb{Q}_p)=0$ and that $\mathbb{F}_{p^k}((t))$ is a finite extension of $\mathbb{F}_p((t))$, for all prime numbers $p$. Therefore, our theorem follows from Lemma \ref{claim} below.
\end{proof}

For the proof of the next result we shall make use of a famous theorem known as Krasner's Lemma \cite{Kra46} (although first proved by Ostrowski \cite{Ost17}).

\begin{fact}[\cite{Fal08}, page 78]\label{KraL}
Let $F$ be a field complete with respect to a non-Archimedean norm and denote by $F^{\text{alg}}$ its algebraic closure. Let $\lvert~\rvert$ be the canonical extension of the norm of $F$ to $F^{\text{alg}}$. If $x\in F^{\text{alg}}$ is separable and $x=x_1,\ldots,x_d$ are all the roots of its minimal polynomial, then for all $y\in F^{\text{alg}}$ we have that
\[
\lvert y-x\rvert<\min\{\lvert x_i-x\rvert\mid 1\leq i\leq d\}
\]
implies $F(x)\subseteq F(y)$.
\end{fact}

\begin{lemma}\label{claim}
The field $\mathbb{F}_p((t))$ admits (at most) countably many field extensions of any fixed finite degree, for all prime numbers $p$.
\end{lemma}
\begin{proof}
Fix a prime number $p$. Let us denote $\mathbb{F}_p((t))$ by $F$ and take $d\in\N$. Clearly, if $d=1$, then we have only one finite extension of degree $d$. Let us then assume that $d>1$. We first reduce to the case of minimal finite extensions $K$ of $F$, i.e., with no intermediate fields. Indeed, if $F$ admits countably many non-isomorphic minimal extensions of any fixed degree, then suppose that we have (up to isomorphism) uncountably many extensions $K_i$ of $F$ of some fixed degree $n$. We can assume that $n$ is minimal, i.e., $F$ admits countably many non-isomorphic extensions of any fixed degree $<n$. For each $i$, we can find a minimal extension $F\subseteq L_i\subseteq K_i$ of degree, say, $1<n_i<n$ over $F$. By assumption, the collection (up to isomorphism) of the $L_i$ is countable, but then the collection of the $K_i$ would also be countable by induction on $n$.

Let us then take a minimal extension $K$ of $F$ of degree $d$. By minimality we must have $K=F(x)$ for some $x$ with minimal polynomial $p(X)$ of degree $d$. Assume first that there is $q\in F[X]$ such that $p(X)=q(X^p)$. Then $F\subseteq F[x^p]\subseteq K$ implies by minimality that $x^p\in F$ and thus $K\simeq\mathbb{F}_p((t^\frac{1}{p}))$. By \cite[Proposition 4.6 (1)]{Mor96}, we can now assume without loss of generality that $K|F$ is separable. We make the following:
\begin{claim}\label{toppol}
Let $F[X]_d$ be the $d+1$-dimensional $F$-vector space of polynomials of degree at most $d$, endowed with the topology induced from $F$ and let $S_d\subseteq F[X]_d$ be the subset of monic polynomials of degree precisely $d$. Endow $S_d$ with the subspace topology. For all separable and irreducible $p\in S_d$ there exists an open set $O_p$ in $S_d$ such that $p\in O_p$ and for all $q\in O_p$ we have that $F[X]/(q)\simeq F[X]/(p)$ as rings. In particular, every element of $O_p$ is irreducible.
\end{claim}
\begin{proof}[Proof of Claim \ref{toppol}]
Let $p\in S_d$ be separable and irreducible, $K:=F[X]/(p)$ and $x_1,\ldots,x_d$ the roots of $p$ is the algebraic closure of $F$. Set
\[
r:=\frac{1}{2}\inf\left\{\lvert x_i-x_j\rvert\mid 1\leq i<j\leq d\right\}
\]
and 
\[
O_p:=\{q\in S_d\mid q(x_i)<r^d\text{ for all }i=1,\ldots,d\}.
\]
Take $q\in O_p$ and let $y_1,\ldots,y_k$ be its roots, for some $k\leq d$. Fix $1\leq i\leq d$ and suppose that $\lvert x_i-y_j\rvert\geq r$ for all $1\leq j\leq k$. Since $q$ is monic we obtain that
\[
\lvert q(x_i)\rvert=\prod_{j=1}^k\lvert x_i-y_j\rvert\geq r^d,
\]
a contradiction. We conclude that, for every $1\leq i\leq d$ and every $q\in O_p$ there is a root $y_j$ of $q$ such that $\lvert x_i-y_j\rvert<r$. If $y$ would be a root such that $\lvert x_i-y\rvert<r$ and $\lvert x_j-y\rvert<r$ for some $1\leq i<j\leq d$, then we would obtain a contradiction from the definition of $r$ as
\[
\lvert x_i-x_j\rvert=\lvert x_i-y+y-x_j\rvert\leq \lvert x_i-y\rvert+\lvert y-x_j\rvert<2r.
\]
Since $O_p\subseteq S_d$, this shows that all $q\in O_p$ are separable. We now enumerate the roots $y_1,\ldots,y_d$ of $q\in O_p$ so that $\lvert x_i-y_i\rvert<r$ for all $1\leq i\leq d$. In particular,
\[
\lvert x_1-y_1\rvert<r<\lvert x_1-x_i\rvert,
\]
for all $1< i\leq d$. Hence, Krasner's Lemma \ref{KraL} yields $K=F(x_1)\subseteq F(y_1)$ and $O_p\subseteq S_d$ now implies that all $q\in O_p$ are irreducible with $K=F[y_1]\simeq F[X]/(q)$. This completes the proof of the claim.
\renewcommand{\qedsymbol}{\#}
\end{proof}
Let $S_d^s$ be the subset of $S_d$ containing all monic separable polynomials of degree $d$. By Claim \ref{toppol}, $S_d^s$ is open in $S_d$. Moreover, if $F$ would have uncountably many minimal separable non-isomorphic extensions, then we would obtain uncountably many disjoint open subsets of $S_d^s\subseteq F[X]_d\simeq F^{d+1}$. However, the latter is excluded by Proposition \ref{obv}. We conclude that $F$ has at most countably many extensions of degree $d$.
\end{proof}

\appendix

\section{More on local fields}
\label{app1}
The facts presented in this appendix are well-known and can easily be found in the literature (e.g.\ in \cite{PE05}). We write them here for the convenience of the reader and to explain notation and terminology.\par
An \emph{Archimedean norm (or absolute value)} on a field $K$ is axiomatised a real-valued function $\lvert~\rvert:K\to[0,+\infty)$ satisfying $\lvert x\rvert=0$ if and only if $x=0$, $\lvert x\cdot y\rvert=\lvert x\rvert\cdot\lvert y\rvert$ and the triangular inequality $\lvert x+y\rvert\leq\lvert x\rvert+\lvert y\rvert$. A \emph{non-Archimedean norm} is an Archimedean norm satisfying the (stronger) \emph{ultrametric triangular inequality} $\lvert x+y\rvert\leq\max(\lvert x\rvert,\lvert y\rvert)$.

Any norm on a field $K$ induces a metric on $K$ by setting the distance between two elements $x$ and $y$ of $K$ to be $\lvert x-y\rvert$. Two norms on a field are \emph{equivalent} if the corresponding metrics induce the same topology on $K$.

All archimedean norms on $\mathbb{Q}$ are equivalent to the standard one $|x|:=\sqrt{x^2}$. The completion of $\mathbb{Q}$ with respect to the metric induced by this norm is $\mathbb{R}$. On the other hand, the set of all non-Archimedean norms of $\mathbb{Q}$ is (up to equivalence) in bijection with the set of prime numbers. This is a theorem of Ostrowski \cite[Theorem 2.1.4(a)]{PE05}. For a prime number $p$, the \emph{$p$-adic norm on }$\mathbb{Q}$ is defined on an integer $a$ as $\lvert a\rvert_p:=p^{-\nu}$, where $\nu$ is the biggest natural number such that $p^\nu$ divides $a$. Then for $a,b\in\mathbb{Z}$, with $b\neq 0$ one naturally sets
\[
\left\lvert\frac{a}{b}\right\rvert_p:=\frac{\lvert a\rvert_p}{\lvert b\rvert_p}.
\]
The field of \emph{$p$-adic numbers} $\mathbb{Q}_p$ can be defined as the completion of $\mathbb{Q}$ with respect to the metric induced by the $p$-adic norm and the $p$-adic norm extends canonically from $\mathbb{Q}$ to $\mathbb{Q}_p$ \cite[Theorem 1.1.4]{PE05}. The finite extensions of $\mathbb{Q}_p$ are the non-Archimedean local fields of characteristic $0$.

Let now $k(t)$ denote the field of rational functions over a field $k$ and $h\in k[t]$ be an irreducible polynomial with coefficients in $k$. The \emph{$h$-adic norm on }$k(t)$ is defined analogously: for a polynomial $f\in k[t]$ set $\lvert f\rvert_h:=h^{-\nu}$, where where $\nu$ is the biggest natural number such that $h^\nu$ divides $f$. Then for $f,g\in k[X]$, with $g\neq 0$ one naturally sets
\[
\left\lvert\frac{f}{g}\right\rvert_h:=\frac{\lvert f\rvert_p}{\lvert g\rvert_p}.
\]
Clearly, $\lvert a\rvert_h=1$ for all $a\in k^\times$. By \cite[Theorem 2.1.4(b)]{PE05}, any norm on $k(t)$ with the latter property is (up to equivalence) $h$-adic for some irreducible $h\in k[t]$, unless it is the \emph{degree norm}:
\[
\left\lvert\frac{f}{g}\right\rvert_h:=\exp(\deg(f)-\deg(g)),
\]
where there is no preferred choice for the base of the exponential function $\exp$ (but all bases give equivalent norms). The field $k((t))$ of \emph{formal Laurent series} over $k$ can be defined as the completion of $k(t)$ with respect to the metric induced by the $t$-adic norm, which again extends canonically from $k(t)$ to $k((t))$. For $k=\mathbb{F}_{p^n}$, where $p$ is prime and $n\in\N$, we obtain the non-Archimedean local fields of characteristic $p$.\par
\begin{comment}
For $k$ being $\mathbb{F}_p((t))$ or $\mathbb{Q}_p$ we now set $v(x):=-\log_p\lvert x\rvert\in\mathbb{Z}$, for all $x\in k$. In \cite[Pages 19,20]{PE05} it is shown that
\[
\mathfrak{O}_k:=\{x\in F\mid v(x)\geq 0\}
\]
is a subring of $F$, with a unique maximal ideal:
\[
\mathfrak{m}_k:=\{x\in F\mid v(x)>0\},
\]
such that the quotient field $\mathfrak{O}_k/\mathfrak{m}_k$ is the finite field with $p$ elements $\mathbb{F}_p$. This is called the \emph{residue field} of $k$.
Now, if $F$ is a non-Archimedean local field, then $F$ is a finite extension of $k$ and the same as above is true replacing $F$ with $k$. Moreover, there is a natural embedding of the corresponding quotient fields:
\[
\mathfrak{O}_k/\mathfrak{m}_k\to\mathfrak{O}_F/\mathfrak{m}_F.
\]
Up to this embedding, $\mathfrak{O}_F/\mathfrak{m}_F$ is a field extension of $\mathbb{F}_p$. By e.g.\ \cite[page 66, F1]{Fal08}, this is a finite extension and its degree $f(F|k)$ is called the \emph{residue degree} of the extension $F|k$. Let $K$ be a finite extension of $F$ of degree $n\in\N$. Similarly, as above we can define the residue degree $f(K|F)$ of the extension $K|F$. The following positive integer is called the \emph{ramification index} of the extension $K|F$:
\[
e(K|F):=\frac{[F:K]}{f(K|F)}
\]
and we refer to it in the proof of Theorem \ref{mainth}. For more details the reader may consult e.g.\ \cite[Sections 24,25]{Fal08}.
\end{comment}

\section{The standard topology on the finite-dimensional projective spaces over a local field} \label{appB}

As a general reference for topology let us mention \cite{Mun00}. As the previous appendix, also here all the presented results are (especially in the Archimedean case) well-known and the reasons to write this appendix are essentially the same as those motivating the previous one.\par 
Here, $(F,\tau)$ will denote a local field endowed with the topology defined by its canonical norm. By definition $(F,\tau)$ is a Hausdorff and locally compact topological space. 
Take $n\in\N$. Since any product of Hausdorff spaces endowed with the product topology is Hausdorff, we have that $F^{n+1}$ is a Hausdorff space, whence so is $F^{n+1}\setminus\{\bar 0\}$.

To prove that $\mathbb{P}^n_F$ is also a Hausdorff space we can make use of the following general result.

\begin{lemma}
\label{haus}
Let $(X,\tau)$ be a topological space and $\sim$ an equivalence relation on $X$. Assume that the canonical map $\pi:X\to X/{\sim}$ is open, where $X/{\sim}$ is endowed with the quotient topology. Then $X/{\sim}$ is Hausdorff if and only if $\{(x,y)\in X\times X\mid x\sim y\}$ is closed in $X\times X$ with respect to the product topology.
\end{lemma}

\begin{proof}
Assume that $X/{\sim}$ is Hausdorff. Let $A$ be the complement of $\{(x,y)\in X\times X\mid x\sim y\}$ in $X\times X$. If $(x,y)\in A$, then $\pi(x)$ and $\pi(y)$ are distinct in $X/{\sim}$. By assumption there are disjoint open sets $O_x,O_y$ in $X/{\sim}$ with $\pi(x)\in O_x$ and $\pi(y)\in O_y$. It follows that $\pi^{-1}(O_x)\times\pi^{-1}(O_y)$ is open in $X\times X$ containing $(x,y)$. Moreover, since $O_x\cap O_y=\emptyset$, we obtain that $\pi^{-1}(O_x)\times\pi^{-1}(O_y)\subseteq A$. This shows that $A$ is open and proves one implication.\par
Conversely, if $x,y\in X$ are such that $\pi(x)\neq\pi(y)$, then $(x,y)\in A$. Thus, if $A$ is open in $X\times X$, there is a subset $O$ of $A$ that is open in $X\times X$ and contains $(x,y)$. Let $p_i:X\times X\to X$ be the canonical projection on the $i$-th coordinate ($i=1,2$). Since $p_1$ and $p_2$ are open maps, we obtain that $p_i(O)$ is open in $X$ for each $i=1,2$. Now, since $\pi$ is an open map by assumption, $\pi(p_i(O))$ is open in $X/{\sim}$ and $\pi(x)\in \pi(p_1(O))$, while $\pi(y)\in \pi(p_2(O))$. If $\pi(z)\in \pi(p_1(O))\cap \pi(p_2(O))$ for some $z\in X$, then $(z,z)\in O\setminus A$, while $O\subseteq A$. This contradiction shows that $\pi(x)$ and $\pi(y)$ can be separated by disjoint open sets in $X/{\sim}$, which is then a Hausdorff space.
\end{proof}

If we now consider the continuous map $f:F^{n+1}\setminus\{\bar 0\}\times F^{n+1}\setminus\{0\}\to F$ defined as
\[
f(x_1,\ldots,x_{n+1};y_1,\ldots,y_{n+1})=\sum_{i\neq j}(x_iy_j-x_jy_i),
\]
then we obtain that $f(\bar x;\bar y)=0$ if and only if $\bar y=a\bar x$ for some $a\in F^\times$. It follows that
\[
f^{-1}(\{0\})=\{(\bar x;\bar y)\mid \bar x\sim \bar y\}
\]
is closed in $F^{n+1}\setminus\{\bar 0\}\times F^{n+1}\setminus\{0\}$. On the other hand, since scalar multiplication by any $a\in F^\times$ is a homeomorphism on $F^{n+1}\setminus\{\bar 0\}$, we obtain that the canonical epimorphism $F^{n+1}\setminus\{\bar 0\}\to\mathbb{P}_F^n$ is open. That is how Lemma \ref{haus} above can be applied to see that $\mathbb{P}_F^n$ is a Hausdorff space with respect to its standard topology.

The reader may have noticed that the assumption that $F$, and so $F^{n+1}$ and $F^{n+1}\setminus\{\bar 0\}$, are Hausdorff has not been used to prove that $\mathbb{P}_F^n$ is a Hausdorff space. However, it is not difficult to see that the only non-Hausdorff toplogical fields have the trivial topology, in which case Lemma \ref{haus} shows that the finite-dimensional projective spaces are not Hausdorff as well. We have proved:

\begin{proposition}
Finite-dimensional projective spaces over topological fields with a non-trivial topology are Hausdorff spaces with respect to the standard topology.    
\end{proposition}

Now, by observing that $F$ is a locally compact complete valued field, we obtain that $F^{n+1}$ is locally compact by e.g.\ \cite[Proposition 6.2.4]{Gou20} and it follows from the fact that $F^{n+1}$ and $\{\bar 0\}$ are closed in $F^{n+1}$ (i.e., $F^{n+1}\setminus\{\bar 0\}$ is locally closed) that $F^{n+1}\setminus\{\bar 0\}$ is locally compact with respect to the subspace topology. 

Let now $\pi:F^{n+1}\setminus\{\bar 0\}\to\mathbb{P}_F^n$ denote the canonical epimorphism and consider the $n$-dimensional sphere in $F^{n+1}$:
\[
\mathbb{S}_F^n:=\left\{\bar x\in F^{n+1} \bigg\vert \sum_{i=1}^{n+1} |x_i|=1\right\},
\]
endowed with the subspace topology of the product topology on $F^{n+1}$.  Since $\pi{\restriction}\mathbb{S}_F^n$ is still a surjective and continuous map, if we show that $\mathbb{S}_F^n$ is compact in $F^{n+1}$, then we will obtain that $\mathbb{P}_F^n$ is a compact space with respect to the standard topology. We have that $F^{n+1}$ is a (complete) metric space with respect to the distance induced by the canonical norm $\lvert~\rvert$ of $F$:
\[
\lVert x-y\rVert:=\sum_{i=1}^{n+1}\lvert x_i-y_i\rvert.
\]
By a \emph{closed ball} in $F^{n+1}$ we mean a set of the form:
\[
B[\bar c,r]:=\{\bar x\in F^{n+1}\mid \lVert x-c\rVert\leq r\}\quad\quad(\bar c\in F^{n+1}, r\in\mathbb{R}).
\]
Since for all $a\in F^\times$ the map
\begin{align*}
m_a:F^{n+1}&\to F^{n+1}\\
\bar x&\mapsto a\bar x
\end{align*}
is continuous and $a\bar 0=\bar 0$, if $B$ is a closed ball with radius $r$ containing $0$, then $m_a(B)$ is a closed ball of radius $\lvert a\rvert r$ containing $0$. Since $F^{n+1}$ is locally compact, there exists a compact set $C_0\subseteq F^{n+1}$ which contains an open set $O_0$ where $\bar 0$ lies. If
\[
d:=\inf\{d(\bar x,\bar 0)\mid \bar x\in C_0\},
\]
then we obtain that $d>0$ because $C_0$ contains the open set $O_0$ and $B:=B[\bar 0,d]$ is a closed ball contained in the compact set $C_0$, thus itself compact. Even if $d<1$, then by the Archimedean property of $\mathbb{R}$ we may find $a\in F$ such that $\lvert a\rvert d\geq 1$. Hence, we obtain that $\mathbb{S}^n\subseteq m_{a}(B)$ and now the compactness of $\mathbb{S}^n$ follows from its closedness. We have proved Fact \ref{Pcomp}.

Before closing this appendix let us observe that by their very definition, $\mathbb{Q}_p$ and $\mathbb{F}_p((t))$ are \emph{separable} spaces with respect to the topology induced by their canonical norms. Indeed, $\mathbb{Q}$ and $\mathbb{F}_p(t)$ are countable dense subsets of $\mathbb{Q}_p$ and $\mathbb{F}_p((t))$, respectively. Hence (as in the case of $\mathbb{R}$ and $\mathbb{C}$, or any other separable metric space), if $F$ is a local field and $D\subseteq F$ is a countable dense subset, then one obtains that:
\[
\{B[\bar c,r]\mid \bar c\in D^n,r\in\mathbb{Q}\}
\]
is a countable base for the topology of $F^n$, for all $n\in\N$. From this we easily deduce:
\begin{proposition}\label{obv}
Let $F$ be a local field and $n\in\N$. Then any family of pairwise disjoint open sets in $F^n$, with respect to the product topology, is finite or countable.
\end{proposition}
\begin{proof}
Let $\mathcal{O}$ be a family of pairwise disjoint open sets in $F^n$ and $\{B_i\mid i\in\N\}$ a countable base for the topology. By definition of base, for any $O\in\mathcal{O}$ there exists some $i\in\N$ such that $B_i\subseteq O$. Consider the function $f:\mathcal{O}\to\mathbb{N}$ mapping any $O\in\mathcal{O}$ to $\min\{i\in\N\mid B_i\subseteq O\}$. Since the elements of $\mathcal{O}$ are pairwise disjoint, $f$ is injective proving that $\mathcal{O}$ must be finite or countable.
\end{proof}
\bibliography{biblio}
\bibliographystyle{abbrv}
\end{document}